\documentclass[11pt,reqno]{amsart}
\usepackage{diagrams}
\usepackage{amssymb, mathtools}
\usepackage[pagewise]{lineno}
\numberwithin{equation}{section}
\usepackage{cite}
\usepackage{url}
\usepackage{graphicx}

\theoremstyle{definition}
\newtheorem{thm}{Theorem}[section]
\newtheorem{cor}[thm]{Corollary}

\newtheorem{exa}[thm]{Example}

\DeclareMathOperator{\Hc}{\mathcal{H}om}

\DeclareMathOperator{\I}{\mathcal{I}}
\DeclareMathOperator{\mo}{\mathcal{O}}

\makeatletter
\@namedef{subjclassname@2020}{\textup{2020} Mathematics Subject Classification}
\makeatother

\newcommand{\mr}[1]{\mathrm{#1}}
\newcommand{\mb}[1]{\mathbb{#1}}

\newcommand{\mc}[1]{\mathcal{#1}}

\newcommand{\mf}[1]{\mathfrak{#1}}

\begin{document}

 \title{Examples of non-flat bundles of rank one}

\author{Ananyo Dan}

\address{School of Mathematics and Statistics, University of Sheffield, Hicks building, Hounsfield Road, S3 7RH, UK}

\email{a.dan@sheffield.ac.uk}

\author{Agust\'in Romano-Vel\'azquez}

\address{Alfréd Rényi Institute Of Mathematics, Hungarian Academy Of Sciences, Reáltanoda Utca 13-15, H-1053, Budapest, Hungary} 

\email{agustin@renyi.hu}

\thanks{}



\date{\today}

\begin{abstract}
 In this article we give several examples of line bundles on 
 certain non-compact surfaces that cannot be equipped with a flat 
 connection. 
\end{abstract}

 \maketitle
 
 \section{Introduction}
 Existence of connections on modules defined over isolated surface singularities 
 have been extensively studied (see \cite{erik1, erik2, gus1, erik3, blo}). However, in 
 these literatures the authors stress that they cannot produce a single
 example of a maximal Cohen-Macaulay (MCM) module over a surface singularity that 
 does not admit a flat connection (see \cite[p. $106$]{gus1}, \cite[p. $1562$]{erik1}, 
 \cite[p. $903$]{blo}), even with the help of a computer (see \cite[p. $322$]{erik2}).
 By flat connection, we mean a connection with zero curvature.
  In this article we produce numerous examples of maximal Cohen-Macaulay modules over 
 certain isolated surface singularities that cannot be equipped with a flat connection.
 In particular, we prove:
 \begin{thm}\label{thm}
 	Let $(X,x)$ be the germ of a normal surface singularity such that the fundamental group of 
 	the link is perfect. Suppose  that $(X,x)$ contains a smooth curve passing through
 	 $x$. Then, there exists a line bundle on $X\setminus \{x\}$ that cannot be equipped with a flat connection. Moreover, one can associate to any such smooth curve (i.e., 
 	passing through $x$), an unique (up to isomorphism) line bundle on $X\setminus \{x\}$ that cannot be equipped with a flat connection.
\end{thm}

As a corollary we produce explicit examples of maximal Cohen-Macaulay
modules that cannot be equipped with a flat connection.

 \begin{cor}\label{cor}
Let $(p,q,r)$ be a triple of positive integers that are pairwise coprime and $r>pq$. Denote by $G(p,q,r)$ the surface in $\mb{C}^3$ defined by the polynomial $X^p+Y^q+Z^r$ and by $U(p,q,r)$ the regular locus of $G(p,q,r)$. Then, 
 there exists line bundles on $U(p,q,r)$ that cannot be equipped with a flat connections.
 \end{cor}

The study of the obstruction to the existence of flat connection on MCM modules 
 has applications in Lie-Rinehart cohomology (see \cite{erik3}) and Chern-Simmons theory (see \cite{blo}).

 {\bf{Acknowledgement}} We thank Prof. Javier F. de Bobadilla and Prof. Duco van Straten for 
their interest in this problem and helpful comments. The first author is funded by 
EPSRC grant number R/$162871-11-1$. The second author is funded by OTKA 126683 and Lend\"ulet 30001. The second author thanks CIRM, Luminy, for its hospitality and for providing a perfect work environment. He also thanks Prof. Javier F. de Bobadilla, the 2021 semester 2 Jean-Morlet Chair, for the invitation.

 \section{Example of non-flat invertible sheaves}
 We give examples of rank $1$ invertible sheaves which cannot be equipped with a flat connection.

 \subsection{Brieskorn-Pham surfaces}
 Given positive integers $(p,q,r)$, denote by $G(p,q,r) \subset \mb{C}^3$
  the zero locus of the 
 polynomial $X^p+Y^q+Z^r$. The resulting surface is the
 \emph{Brieskorn-Pham type surface}. The origin $0$ is the only singularity of the 
 surface.
 Denote by 
 \[\mc{S}:=\{C \subset G(p,q,r)\, |\, C \mbox{ is a non-singular curve passing through the origin }0\}.\]
  The following theorem proves the 
 existence of smooth curves through the singularity of $G(p,q,r)$.
  \begin{thm}\label{thm:exist}
  Assume $(p,q,r)=1$ and $p \le q \le r$.  Then, $\mc{S} \not= \emptyset$ if and 
  only if at least one of the following conditions hold:
  \begin{enumerate}
   \item two of the three integers $(p,q,r)$ are coprime and the other one 
   is divisible by at least one of the two coprime numbers,
   \item the inequality $r>pq/\mr{gcd}(p,q)$ holds.
  \end{enumerate}
 \end{thm}

 \begin{proof}
  See \cite[Theorem $3$]{jiang}.
 \end{proof}

  \subsection{Proof of Theorem \ref{thm}} 
   Denote by $U$ the regular locus of $(X,x)$, i.e. $U= X\setminus \{x\}$. Note that the fundamental group $\pi_1(U)$ is the same as the fundamental group of the link $L$ of $(X,x)$. By hypothesis, the fundamental group of $L$ is perfect. This means that the abelianization $\pi_1(U)^{\mr{ab}}$ of 
   the fundamental group $\pi_1(U)$ is trivial.  Since $\mr{GL}_1(\mb{C})=\mb{C}^*$ is abelian, any $1$-dimensional
   group representation of $\pi_1(U)$ factors through  $\pi_1(U)^{\mr{ab}}$,
   which is trivial. Therefore, every $1$-dimensional representation of 
$\pi_1(U)$ is trivial. By the Riemann-Hilbert correspondence, this implies
that there does not exist any non-trivial line bundle on $U$ 
that can be equipped with a flat connection. Therefore, it suffices to show the 
existence of non-trivial line bundles on $U$.

By hypothesis, there exists  $C \subset X$ a smooth curve contained in the surface $X$ passing through the singular point. Denote by $\I_C$ the ideal sheaf of $C$. We now show that the restriction $M_U$ of $M:=\Hc_X(\I_C,\mo_X)$ to 
$U$ is a non-trivial line bundle.
Indeed, $M_U$ is trivial if and only if $i_*(M_U) \cong M$ is trivial (invertible 
sheaves are reflexive, thereby determined uniquely by its 
restriction to the open subset $U$), where $i: U \to X$ is the 
natural inclusion. Furthermore, since $M$ is a reflexive module, 
$M$ is trivial if and only if $M^{\vee}$ is trivial (double dual 
of a reflexive module is isomorphic to itself).
Consider the short exact sequence:
\begin{equation}\label{eq01}
 0 \to \I_C \to \mo_X \to \mo_C \to 0
\end{equation}
By the depth comparison in exact sequences, we have that 
$\I_C$ is a reflexive $\mo_X$-module. This implies that 
\[M^{\vee} \cong \I_C^{\vee \vee} \cong \I_C.\]
This implies, $M_U$ is trivial if and only if $\I_C$ is trivial.
By \eqref{eq01}, $\I_C$ is trivial if and only if $C$ is a Cartier divisor.
So, it suffices to show that $C$ is a Weil divisor which is not Cartier.

We prove this by contradiction. In particular, suppose that $C$ is 
a Cartier divisor. We are going to give a contradiction to the non-smoothness 
of $X$. 
Let $f \in \mo_X$ be a function defining $C$ (the existence of $f$ is guaranteed as 
$C$ is Cartier) and $\mf{m}$ (resp. $\mf{m}'$) the maximal ideal of $\mo_X$ (resp. $\mo_X/(f)$).
We then have the following short exact sequence of $\mo_X$-modules:
\begin{equation}\label{eq03}
	0 \to \mf{m} \cap (f) \to \mf{m} \to \mf{m}' \to 0.
\end{equation}
Since $C$ is a smooth curve, the dimension as a complex vector space of $\mf{m}'/(\mf{m}')^2$ is one. 
Let $c' \in \mf{m}'$ be a generator of $\mf{m}'/(\mf{m}')^2$.  By the surjectivity of 
\eqref{eq03}, there exists an element $c \in \mf{m}$ that is mapped onto $c'$. 
We will now show that $\mf{m}/(\mf{m})^2$ is generated as a $\mb{C}$-vector space by $f$ and $c$.
Indeed, let $g$ be any element in $\mf{m}$. If $g$ maps to $0$ in $\mf{m}'$ then by the exact sequence 
\eqref{eq03}, $g=\beta f$ for some $\beta \in \mo_X$. If $g$ maps to a non-zero element in $\mf{m}'$,
then there exists $\alpha \in \mo_X$ such that $g-\alpha c$ is mapped to zero in $\mf{m}'$.
Hence by the exactness of \eqref{eq03}, there exists $\beta \in \mo_X$ such that $\beta f=g -\alpha c$.  
Therefore, $g$ is a $\mo_X$-linear combination of $f$ and $c$ in both cases. 
Since $g$ was arbitrary, this implies that 
the dimension as a complex vector space of  $\mf{m}/(\mf{m})^2$ is two (generated by $f$ and $c$). 
But this will mean $X$ is smooth at the origin, which is a contradiction.
Hence, $C$ cannot by a Cartier divisor. This proves the theorem.\qed

\begin{proof}[Proof of Corollary \ref{cor}]
 	Let $(p,q,r)$ be a triple of positive integers that are pairwise coprime. By 
 	\cite[p.~7]{N99}, the link of $(G(p,q,r),0)$ is an integer homology sphere.
 	Hence, the fundamental group of the link is perfect.
 	By Theorem \ref{thm:exist} there exists a smooth curve in $G(p,q,r)$ passing 
 	through the origin. The corollary then follows from Theorem \ref{thm}.
\end{proof}

\begin{exa}
 Let $f= x^2 + y^3 + z^7$. By Corollary \ref{cor}, 
 the hypersurface in $\mb{C}^3$ defined by $f$ admits MCM modules of rank one that cannot be 
 equipped with a flat connection.
\end{exa}


\begin{thebibliography}{1}

\bibitem{blo}
S.~Bloch and H.~Esnault.
\newblock Algebraic {Chern-Simons} theory.
\newblock {\em American Journal of Mathematics}, 119(4):903--952, 1997.

\bibitem{erik2}
E.~Eriksen and T.~S. Gustavsen.
\newblock Computing obstructions for existence of connections on modules.
\newblock {\em Journal of Symbolic Computation}, 42(3):313--323, 2007.

\bibitem{erik1}
E.~Eriksen and T.~S. Gustavsen.
\newblock Connections on modules over singularities of finite {CM}
  representation type.
\newblock {\em Journal of Pure and Applied Algebra}, 212(7):1561--1574, 2008.

\bibitem{gus1}
E.~Eriksen and T.~S. Gustavsen.
\newblock Connections on modules over singularities of finite and tame {CM}
  representation type.
\newblock In {\em Generalized Lie Theory in Mathematics, Physics and Beyond},
  pages 99--108. Springer, Berlin, Heidelberg, 2009.

\bibitem{erik3}
E.~Eriksen and T.~S. Gustavsen.
\newblock {Lie-Rinehart} cohomology and integrable connections on modules of
  rank one.
\newblock {\em Journal of Algebra}, 322(12):4283--4294, 2009.

\bibitem{jiang}
G.~Jiang, M.~Oka, T.~P. Duc, and D.~Siersma.
\newblock Lines on {B}rieskorn-{P}ham surfaces.
\newblock {\em Kodai Mathematical Journal}, 23(2):214--223, 2000.

\bibitem{N99}
A.~{N\'emethi}.
\newblock {Five lectures on normal surface singularities}.
\newblock In {\em Low dimensional topology. Proceedings of the summer school,
  Budapest, Hungary, August 3--14, 1998}, pages 269--351. Budapest: J\'anos
  Bolyai Mathematical Society, 1999.

\end{thebibliography}
 \end{document}